\documentclass[letterpaper]{article}

\usepackage[left=4cm,right=4cm,bottom=3.3cm,top=2.8cm]{geometry}
\usepackage{amsmath,amsthm,amsfonts,amssymb,amsxtra,bm} 
\usepackage{hyperref}
\usepackage{cite}
\usepackage{graphicx}
\usepackage{enumitem}
\usepackage{cite}
\usepackage{psfrag}
\usepackage{subfig}

\newtheorem{theorem}{Theorem}
\newtheorem{corollary}{Corollary}
\newtheorem{lemma}{Lemma}

\newtheorem{definition}{Definition}

\author{Ondrej Slu\v{c}iak}
\date{}

\title{On Inflection Points of the Lehmer Mean Function}

\hypersetup{
	unicode,
	bookmarksnumbered,
	bookmarksopen=false,
	bookmarksopenlevel=1,
	pdfauthor={O.Slu\v{c}iak},
	pdftitle={On Inflection Points of the Lehmer Mean Function},
	pdfkeywords={},
	pdfsubject={}
}

\newcommand{\ve}[1]{\MakeLowercase{\mathbf{#1}}}
\newcommand{\bfs}[1]{\boldsymbol{#1}}

\begin{document}
\maketitle

\begin{abstract}
We prove that the Lehmer mean function of two or three positive numbers has always one and only one inflection point. We further show that in case of two numbers, the inflection point is $p^\star=1$, and we discuss the location of the inflection point in case of three numbers. We furthermore provide an example of a Lehmer mean function with more than one inflection point and provide simple bounds on the number of inflection points for arbitrary many numbers.
\end{abstract}

\section{Introduction}
Over the years a big effort has been made in the analysis of the power mean function $M(p)$. Being a generalization of classical means\cite{Bullen03}, i.e.,
\[M(p) = \left(\sum_{i=1}^n\omega_ix_i^p\right)^{\frac{1}{p}},\]
it has attracted a lot of attention of mathematicians. The question of convexity of this function turned out to be very challenging\cite{Norris37,Beckenbach42,Shniad48}. It is now known that the power mean function, for $x_1\ne x_2$, has only one inflection point in case of $n=2$, which may be different from zero~\cite{DeLaGrandville06,Nam08}. It is also known that for $n>2$ the function may have more inflection points. However, very little is known about the location of the inflection points as well as the influence of $x_i$ on the behaviour of the function.

Even less is known about the so-called \emph{Lehmer mean} function\footnote{Note that in literature the term ``counter-harmonic mean''\cite{Bullen03}, or ``contraharmonic mean'' is often used. Although ``Lehmer mean'' is usually used for the case $n=2$~\cite{Lehmer71}, we will use this term for general case.} $L(p)$ and its convexity. 
To the best of our knowledge, the question of the number of inflection points has not been analyzed in literature either. In this paper we prove that there is always one and only one inflection point in case $n=2$ and $n=3$. We also provide a bound on the number of inflection points for an arbitrary $n$ and show an example of a Lehmer mean function with three inflection points.

\section{Lehmer mean}

First, let us review the definition of Lehmer mean.

\begin{definition}[Lehmer mean~\cite{Bullen03}]
Let $\ve{x}=\{x_1,x_2,\dots,x_n\}$ be $n$ real non-negative numbers, $\bfs{\omega}=\{\omega_1,\omega_2,\dots,\omega_3\}$ be~$n$ non-negative numbers (weights) and let for any $p\in\mathbb{R}$ hold that $\sum_{i=1}^n\omega_ix_i^{p-1}~\ne~0$. The Lehmer mean is then a function defined as
\begin{equation}\label{eq:genlehm}
L(p;\bfs{\omega},\ve{x}) \triangleq \frac{\sum_{i=1}^n \omega_ix_i^p}{\sum_{i=1}^n \omega_ix_i^{p-1}}.
\end{equation}
\end{definition}
Throughout this paper we are interested in the case when the weights $\omega_i$ are constant. By simplifying the notation, let us consider the Lehmer mean function to be defined as
\begin{equation}\label{eq:lehmf}
L(p) \triangleq \frac{\sum_{i=1}^n x_i^p}{\sum_{i=1}^n x_i^{p-1}}.
\end{equation}

Note that allowing parameters $x_i$ to be zero for some $i$, means that the number $n$ is reduced by the number of zero parameters. Without any inconsistency, we might also simply assume that $x_i$ are all positive.

Let us now recall some important properties.
\begin{lemma}\label{lem:monot}
$L(p)$ is a monotonously increasing function.
\end{lemma}

\begin{proof}
The proof follows from~\cite[p.~246, Theorem~3]{Bullen03}. In general, if a function $f(\cdot)$ is convex then for $x_1\leq y_1$ and $x_2\leq y_2$, it holds that $\frac{f(x_1)-f(x_2)}{x_1-x_2}\leq \frac{f(y_1)-f(y_2)}{y_1-y_2}$. Directly setting $f(t)=\log(\sum_{i=1}^nx_i^t)$, which is a convex function for $x_i> 0$, and if $x_1=p$, $x_2=p-1$, $y_1=s$, and $y_2=s-1$, then for any $p\leq s$,
$\frac{\sum_{i=1}^nx_i^p}{\sum_{i=1}^nx_i^{p-1}}\leq  \frac{\sum_{i=1}^nx_i^s}{\sum_{i=1}^nx_i^{s-1}}$.
\end{proof}
\vspace{0.3cm}

\noindent As a consequence of Lemma~\ref{lem:monot}, we obtain the following inequality.
\begin{corollary}\label{cor:ineq}
For any $x_i> 0$ and $p\in\mathbb{R}$,
\begin{equation}\label{eq:ineq}
\frac{\sum_{i=1}^n x_i^p\log x_i}{\sum_{i=1}^n x_i^p}\geq\frac{\sum_{i=1}^n x_i^{p-1}\log x_i}{\sum_{i=1}^n x_i^{p-1}},
\end{equation}
with equality if and only if all $x_i=x_j$.
\end{corollary}
\vspace{0.1cm}

\begin{proof}
By taking the first derivative, we find
\begin{equation}
L^{\prime}(p) = L(p)\left[\frac{\sum_{i=1}^n x_i^p\log x_i}{\sum_{i=1}^n x_i^p}-\frac{\sum_{i=1}^n x_i^{p-1}\log x_i}{\sum_{i=1}^n x_i^{p-1}}\right]. 
\end{equation}
Since $L(p)> 0$ and from Lemma~\ref{lem:monot} follows that also for the second part it must hold that $\frac{\sum_{i=1}^n x_i^p\log x_i}{\sum_{i=1}^n x_i^p}\geq\frac{\sum_{i=1}^n x_i^{p-1}\log x_i}{\sum_{i=1}^n x_i^{p-1}}$, with equality if and only if $x_i=x_j$, for all $i,j$.
\end{proof}
\vspace{0.3cm}

Besides the monotonicity, as in case of other mean functions, it is also known that the Lehmer mean is a continuous function bounded by two horizontal asymptotes, i.e.,
\[\lim_{p\to-\infty}L(p) = \min_{i=1,2,\dots,n} x_i,\]
and
\[\lim_{p\to\infty}L(p) = \max_{i=1,2,\dots,n} x_i.\]
This means that function $L(p)$ is convex near $p=-\infty$ and concave near $p=\infty$ and thus the Lehmer mean function must have at least one inflection point. However, to the best of our knowledge, the true number of inflection points, or bounds on the number of inflection points have yet not been analyzed in literature.

\section{Inflection points of the Lehmer mean function}

Analogously to power mean functions~\cite{Shniad48}, the second derivative of the Lehmer mean takes the form
\[
L^{\prime\prime}(p)=L(p)\left(\Lambda^{\prime\prime}(p)+(\Lambda^\prime(p))^2\right),
\]
where $\Lambda(p)=\log L(p)$.
After expanding the expression we find that
\begin{multline}\label{eq:secder}
L^{\prime\prime}(p) = L(p)\!\left[\frac{\sum_{i=1}^n x_i^p(\log x_i)^2}{\sum_{i=1}^n x_i^p}-\frac{\sum_{i=1}^n x_i^{p-1}(\log x_i)^2}{\sum_{i=1}^n x_i^{p-1}}\right.\\\left.-2\frac{\sum_{i=1}^n x_i^{p-1}\log x_i}{\sum_{i=1}^n x_i^{p-1}}\left(\frac{\sum_{i=1}^n x_i^p\log x_i}{\sum_{i=1}^n x_i^p}-\frac{\sum_{i=1}^n x_i^{p-1}\log x_i}{\sum_{i=1}^n x_i^{p-1}}\right)\right]\!,
\end{multline}
which is, in general, easier to analyze than the second derivative of the power mean function~\cite{Nam08}. 

\begin{lemma}\label{lem:odd}
The number of inflection points must be an odd number.
\end{lemma}

\begin{proof}Note that since the monotonicity is strict, unless $x_i=x_j$ for all $i,j$, the zeros of the second derivative are always inflection points and never saddle points. This, together with the fact that there are two asymptotes for $p=-\infty$ and $p=\infty$, means that the number of zeros of the second derivative must be always odd. Thus, there can be only one, three, five, etc., inflection points.
\end{proof}

\vspace{0.3cm}
Before stating some general bounds on the number of inflection points based on~\eqref{eq:secder}, let us consider simpler cases.

\subsection[Special case n=2]{Special case $n=2$}\label{subsec:n2}

Let us start with a simple case $n=2$. Unlike the power mean function~\cite{Nam08}, in case of Lehmer mean it is easy to find that Eq.~\eqref{eq:secder} in case $n=2$ takes the form
\begin{equation}
L^{\prime\prime}(p) = x_1(a-1)(\log a)^2a^p\frac{a-a^p}{(a^p+a)^3},
\end{equation}
with $a=\frac{x_1}{x_2}$.

We can thus directly see that the function has only one inflection point in $p=1$ if $x_1\ne x_2$. Interestingly, $L(1)=\frac{x_1+x_2}{2}$, i.e., the Lehmer mean function $L(p)$ changes from convex to concave at the arithmetic mean\footnote{Note that the inflection point of the general Lehmer mean function~\eqref{eq:genlehm} with weights $\omega_1$ and $\omega_2$ is $p^\star=1-\frac{\log\frac{\omega_1}{\omega_2}}{\log\frac{x_1}{x_2}}$. Furthermore, by plugging $p^\star$ into \eqref{eq:genlehm} we observe that the general Lehmer mean changes from convex to concave again at the simple arithmetic mean, i.e., $L(p^\star;\bfs{\omega},\ve{x})=\frac{x_1+x_2}{2}$, and not at the weighted arithmetic mean.} (see~the example on Fig.~\ref{fig:fig1} with $x_1=0.5$ and $x_2=2.5$.). 

\begin{figure}[t!]\centering
\subfloat[][Lehmer mean function $L(p)$.]{\centering
\psfrag{p}[Bc][Bc][0.8][0]{$p$}
\psfrag{L}[Bc][Bc][0.8][0]{\raisebox{0.4cm}{$L(p)$}}
\psfrag{a}[Bc][Bc][0.7][0]{$-4$}
\psfrag{b}[Bc][Bc][0.7][0]{$-3$}
\psfrag{c}[Bc][Bc][0.7][0]{$-2$}
\psfrag{d}[Bc][Bc][0.7][0]{$-1$}
\psfrag{e}[Bc][Bc][0.7][0]{$0$}
\psfrag{f}[Bc][Bc][0.7][0]{$1$}
\psfrag{g}[Bc][Bc][0.7][0]{$2$}
\psfrag{h}[Bc][Bc][0.7][0]{$3$}
\psfrag{i}[Bc][Bc][0.7][0]{$4$}
\psfrag{j}[Bc][Bc][0.7][0]{$5$}
\psfrag{A}[Br][Br][0.7][0]{$0$}
\psfrag{B}[Br][Br][0.7][0]{$0.5$}
\psfrag{C}[Br][Br][0.7][0]{$1$}
\psfrag{D}[Br][Br][0.7][0]{$1.5$}
\psfrag{E}[Br][Br][0.7][0]{$2$}
\psfrag{F}[Br][Br][0.7][0]{$2.5$}
\psfrag{G}[Br][Br][0.7][0]{$3$}
\psfrag{X}[Bc][Bc][0.8][0]{$p^\star=1$}
\includegraphics[width=5.4cm]{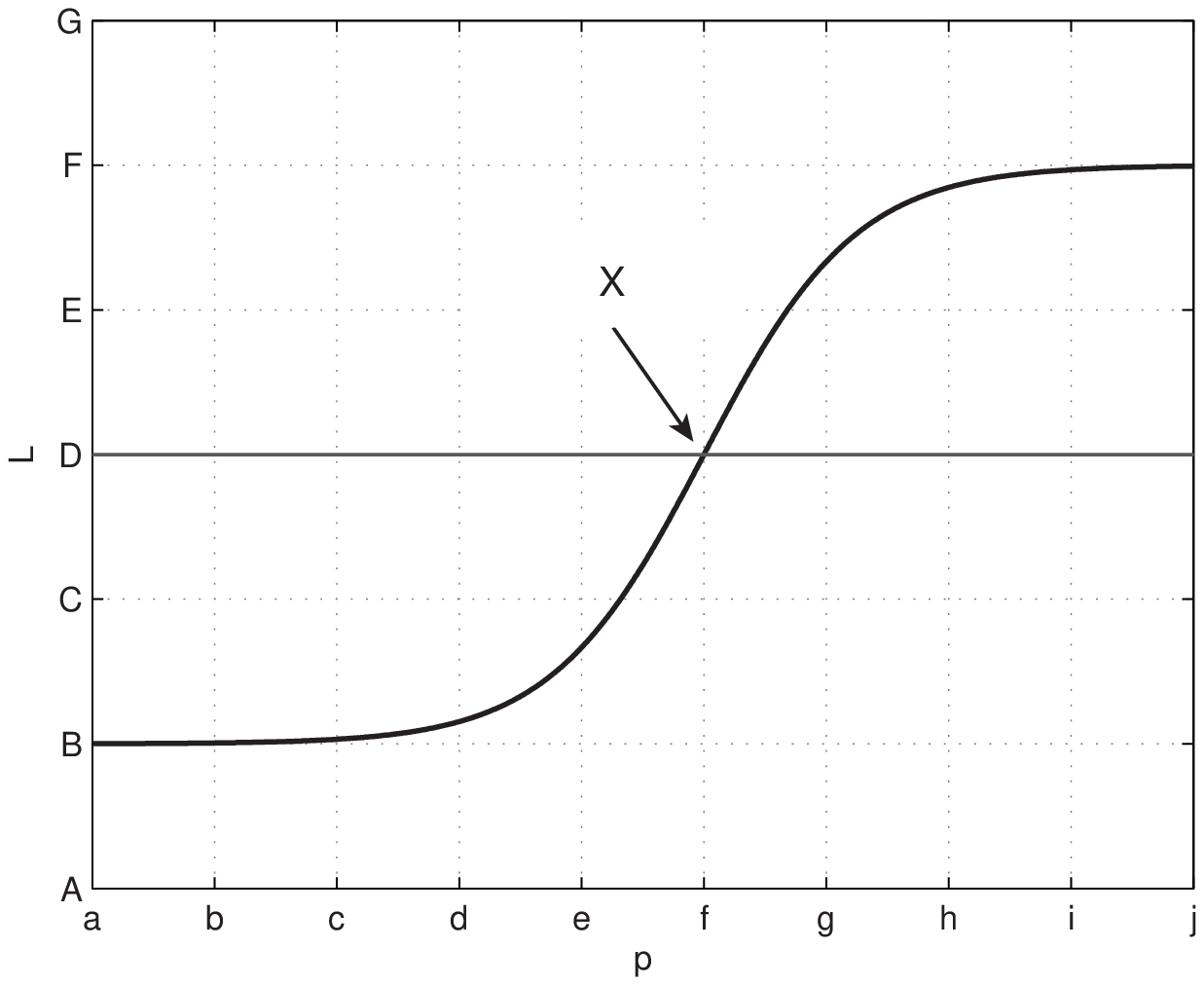}
}\hspace{0.6cm}
\subfloat[][Second derivative of function $L(p)$.]{\centering
\psfrag{p}[Bc][Bc][0.8][0]{$p$}
\psfrag{M}[Bc][Bc][0.8][0]{\raisebox{0.4cm}{$L^{\prime\prime}(p)$}}
\psfrag{a}[Bc][Bc][0.7][0]{$-4$}
\psfrag{b}[Bc][Bc][0.7][0]{$-3$}
\psfrag{c}[Bc][Bc][0.7][0]{$-2$}
\psfrag{d}[Bc][Bc][0.7][0]{$-1$}
\psfrag{e}[Bc][Bc][0.7][0]{$0$}
\psfrag{f}[Bc][Bc][0.7][0]{$1$}
\psfrag{g}[Bc][Bc][0.7][0]{$2$}
\psfrag{h}[Bc][Bc][0.7][0]{$3$}
\psfrag{i}[Bc][Bc][0.7][0]{$4$}
\psfrag{j}[Bc][Bc][0.7][0]{$5$}
\psfrag{A}[Br][Br][0.7][0]{$-0.5$}
\psfrag{B}[Br][Br][0.7][0]{$-0.4$}
\psfrag{C}[Br][Br][0.7][0]{$-0.3$}
\psfrag{D}[Br][Br][0.7][0]{$-0.2$}
\psfrag{E}[Br][Br][0.7][0]{$-0.1$}
\psfrag{F}[Br][Br][0.7][0]{$0$}
\psfrag{G}[Br][Br][0.7][0]{$0.1$}
\psfrag{H}[Br][Br][0.7][0]{$0.2$}
\psfrag{I}[Br][Br][0.7][0]{$0.3$}
\psfrag{J}[Br][Br][0.7][0]{$0.4$}
\psfrag{K}[Br][Br][0.7][0]{$0.5$}
\psfrag{X}[Bl][Bl][0.8][0]{$p^\star=1$}
\includegraphics[width=5.4cm]{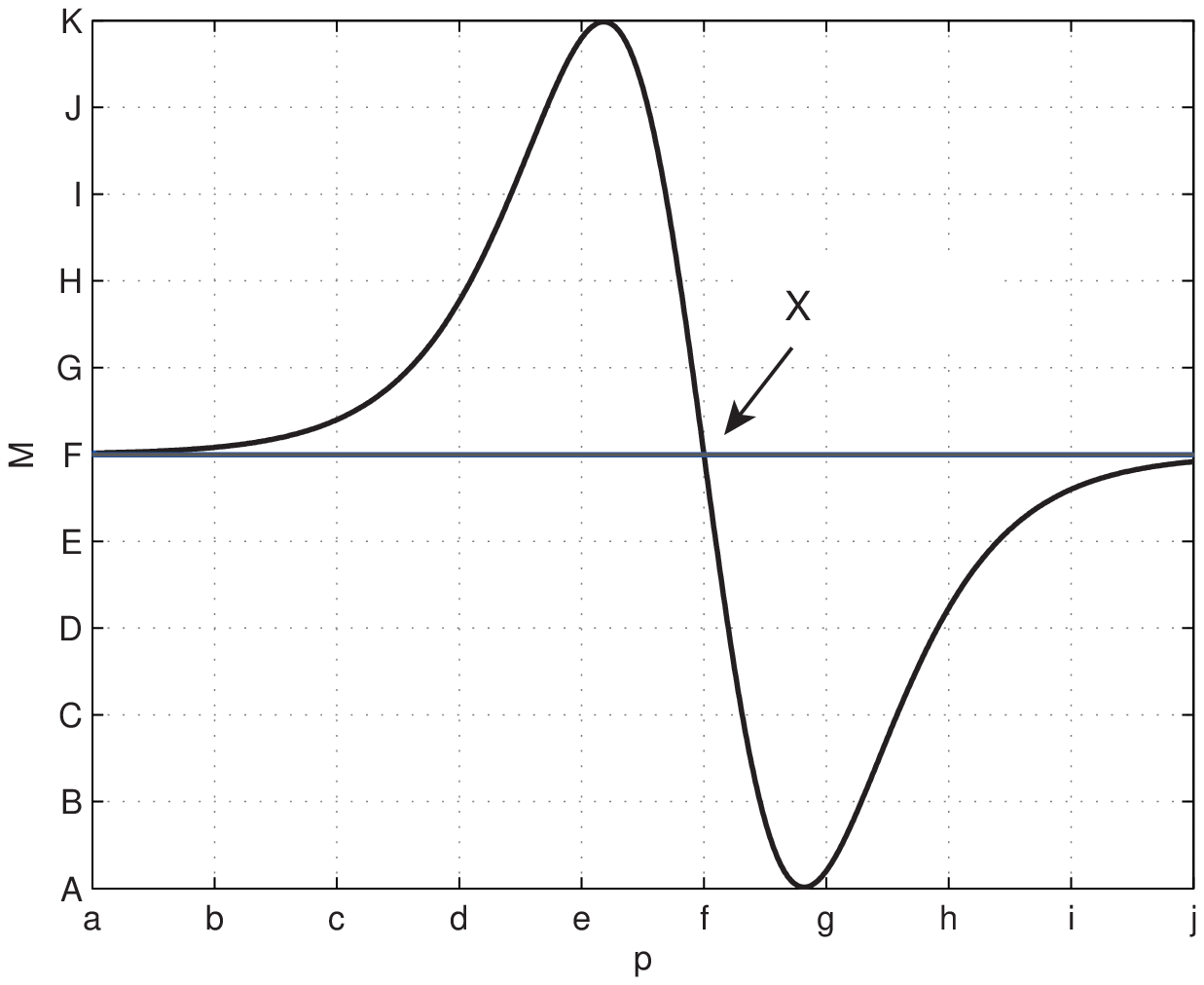}
}
\caption{Example for $n=2$. $L(p) = \frac{0.5^p+2.5^p}{0.5^{p-1}+2.5^{p-1}}$. Inflection point at the arithmetic mean. $p^\star=1$, $L(1)=\frac{0.5+2.5}{2}=1.5$.}\label{fig:fig1}
\end{figure}

\pagebreak
\subsection[Special case n=3]{Special case $n=3$}\label{subsec:n3}

With increasing number $n$, the analysis becomes more complicated. Still feasible is the case of $n=3$. By taking the second derivative and after rearranging the terms, we find that
\begin{align}\label{eq:n3}\textstyle
L^{\prime\prime}(p)=\frac{(x_1x_2x_3)^{p-1}}{\left(x_1^{p-1}+x_2^{p-1}+x_3^{p-1}\right)^3}\Biggl[
&\textstyle\left.\left(\left(\frac{x_2}{x_3}\right)^{p-1}-\left(\frac{x_1}{x_3}\right)^{p-1}\right)(x_1-x_2)\left(\log\frac{x_1}{x_2}\right)^2+\right.\nonumber\\
&\textstyle\left.\left(\left(\frac{x_3}{x_2}\right)^{p-1}-\left(\frac{x_1}{x_2}\right)^{p-1}\right)(x_1-x_3)\left(\log\frac{x_1}{x_3}\right)^2+\right.\nonumber\\
&\textstyle\left(\left(\frac{x_3}{x_1}\right)^{p-1}-\left(\frac{x_2}{x_1}\right)^{p-1}\right)(x_2-x_3)\left(\log\frac{x_2}{x_3}\right)^2+K\ \Biggr],
\end{align}\vspace{-0.2cm}
with 
\begin{equation}\label{eq:K}
K = (x_1-x_2)\log \frac{x_1}{x_2}\log\frac{x_1x_2}{x_3x_3}+(x_1-x_3)\log \frac{x_1}{x_3}\log\frac{x_1x_3}{x_2x_2}+(x_2-x_3)\log \frac{x_2}{x_3}\log\frac{x_2x_3}{x_1x_1}.
\end{equation}

We now show that there is one and only one inflection point in this case, thus, $L(p)$ is concave for $p>p^\star$ and convex for $p<p^\star$. Note that such result is not known in case of power mean functions.

\begin{theorem}
Assuming $n=3$ and $x_1\ne x_2\ne x_3$, the Lehmer mean function $L(p)$ has one and only one inflection point and this inflection point is different from $1$. 
\end{theorem}

\begin{proof}
Since $x_1,x_2,$ and $x_3$ are all positive real numbers, the first part of the right hand side of Eq.~\eqref{eq:n3} is always positive, and therefore only the second part is of interest. Let us denote the right part
\[\begin{array}{rl}
\widetilde{L}(p)=
&\left.\left(\left(\frac{x_2}{x_3}\right)^{p-1}-\left(\frac{x_1}{x_3}\right)^{p-1}\right)(x_1-x_2)\left(\log\frac{x_1}{x_2}\right)^2+\right.\\
&\left.\left(\left(\frac{x_3}{x_2}\right)^{p-1}-\left(\frac{x_1}{x_2}\right)^{p-1}\right)(x_1-x_3)\left(\log\frac{x_1}{x_3}\right)^2+\right.\\
&\left(\left(\frac{x_3}{x_1}\right)^{p-1}-\left(\frac{x_2}{x_1}\right)^{p-1}\right)(x_2-x_3)\left(\log\frac{x_2}{x_3}\right)^2 +K.
\end{array}\]
It is then sufficient to show that $\widetilde{L}(p)$ is a stricly monotonous function. By taking the derivative and rearranging the terms we find the derivative
\begin{align}\label{eq:tildLder}\textstyle
{\widetilde{L}}^\prime(p)=
&\textstyle\log\frac{x_2}{x_1}\log\frac{x_2}{x_3}\left(\left(\frac{x_2}{x_3}\right)^{p-1}(x_1-x_2)\log\frac{x_2}{x_1}+\left(\frac{x_2}{x_1}\right)^{p-1}(x_2-x_3)\log\frac{x_3}{x_2}\right)+\nonumber\\
&\textstyle\log\frac{x_2}{x_1}\log\frac{x_3}{x_1}\left(\left(\frac{x_1}{x_3}\right)^{p-1}(x_1-x_2)\log\frac{x_2}{x_1}+\left(\frac{x_1}{x_2}\right)^{p-1}(x_1-x_3)\log\frac{x_3}{x_1}\right)+\nonumber\\
&\textstyle\log\frac{x_3}{x_1}\log\frac{x_3}{x_2}\left(\left(\frac{x_3}{x_2}\right)^{p-1}(x_1-x_3)\log\frac{x_3}{x_1}+\left(\frac{x_3}{x_1}\right)^{p-1}(x_2-x_3)\log\frac{x_3}{x_2}\right).
\end{align}
\begin{figure}[t!]\centering
\subfloat[][Second derivative of $L(p)$.]{\centering
\psfrag{p}[Bc][Bc][0.8][0]{$p$}
\psfrag{L}[Bc][Bc][0.8][0]{\raisebox{0.6cm}{$L^{\prime\prime}(p)$}}
\psfrag{a}[Bc][Bc][0.7][0]{$-10$}
\psfrag{b}[Bc][Bc][0.7][0]{$-8$}
\psfrag{c}[Bc][Bc][0.7][0]{$-6$}
\psfrag{d}[Bc][Bc][0.7][0]{$-4$}
\psfrag{e}[Bc][Bc][0.7][0]{$-2$}
\psfrag{f}[Bc][Bc][0.7][0]{$0$}
\psfrag{g}[Bc][Bc][0.7][0]{$2$}
\psfrag{h}[Bc][Bc][0.7][0]{$4$}
\psfrag{i}[Bc][Bc][0.7][0]{$6$}
\psfrag{j}[Bc][Bc][0.7][0]{$8$}
\psfrag{k}[Bc][Bc][0.7][0]{$10$}
\psfrag{A}[Br][Br][0.7][0]{$-0.15$}
\psfrag{B}[Br][Br][0.7][0]{$-0.1$}
\psfrag{C}[Br][Br][0.7][0]{$-0.05$}
\psfrag{D}[Br][Br][0.7][0]{$0$}
\psfrag{E}[Br][Br][0.7][0]{$0.05$}
\psfrag{F}[Br][Br][0.7][0]{$0.1$}
\psfrag{G}[Br][Br][0.7][0]{$0.15$}
\psfrag{H}[Br][Br][0.7][0]{$0.2$}
\psfrag{I}[Br][Br][0.7][0]{$0.3$}
\psfrag{J}[Br][Br][0.7][0]{$0.4$}
\psfrag{K}[Br][Br][0.7][0]{$0.5$}
\psfrag{X}[Bl][Bl][0.8][0]{$p^\star=0.707$}
\includegraphics[width=5.4cm]{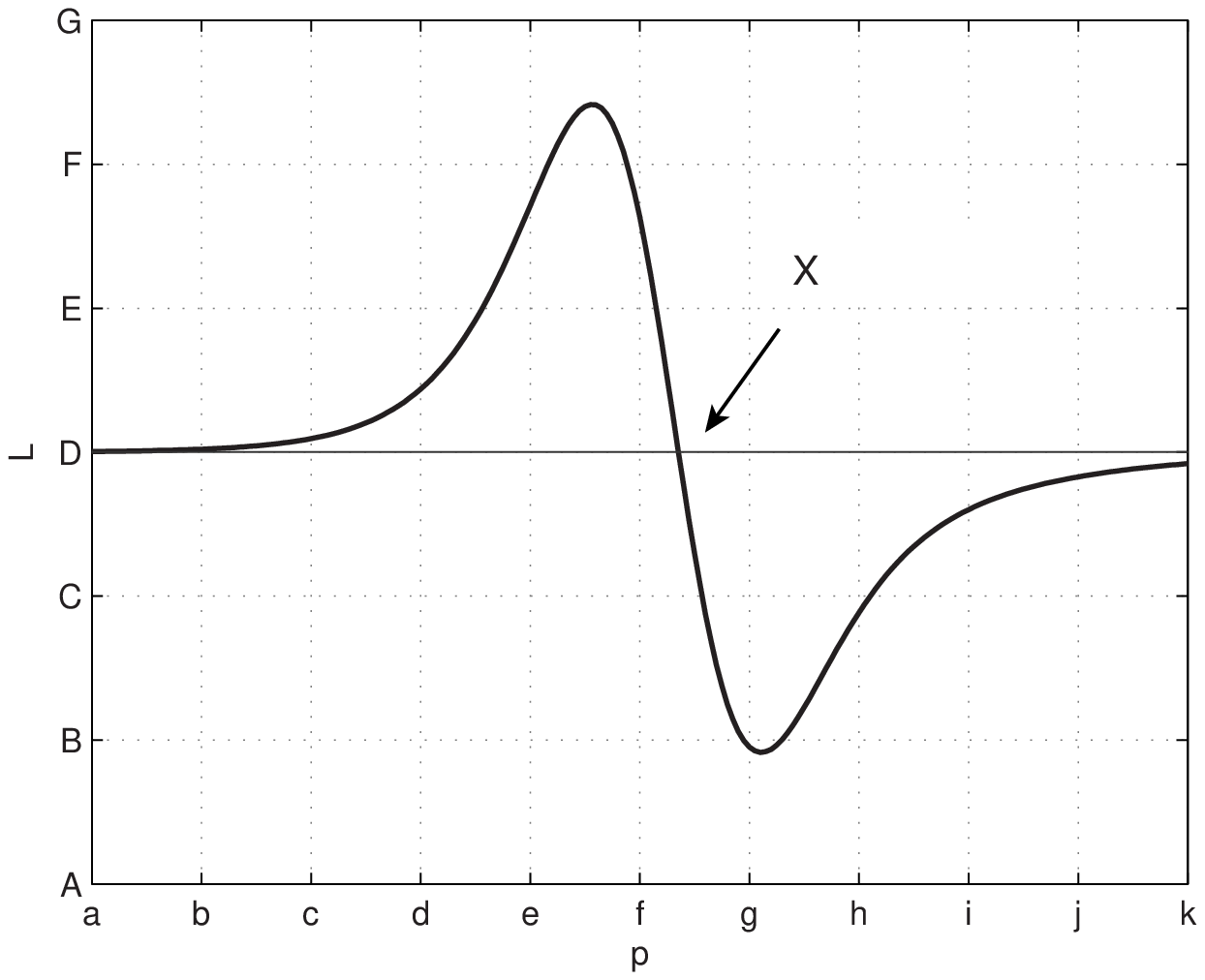}
}\hspace{0.6cm}
\subfloat[][Function $\widetilde{L}(p)$. Since $K\!=\!-0.94$, the inflection point is smaller than one.]{\centering
\psfrag{p}[Bc][Bc][0.8][0]{$p$}
\psfrag{L}[Bc][Bc][0.8][0]{\raisebox{0.5cm}{$\widetilde{L}(p)$}}
\psfrag{a}[Bc][Bc][0.7][0]{$-2$}
\psfrag{b}[Bc][Bc][0.7][0]{$-1$}
\psfrag{c}[Bc][Bc][0.7][0]{$0$}
\psfrag{d}[Bc][Bc][0.7][0]{$1$}
\psfrag{e}[Bc][Bc][0.7][0]{$2$}
\psfrag{f}[Bc][Bc][0.7][0]{$3$}
\psfrag{g}[Bc][Bc][0.7][0]{$4$}
\psfrag{A}[Br][Br][0.7][0]{$-10$}
\psfrag{B}[Br][Br][0.7][0]{$-8$}
\psfrag{C}[Br][Br][0.7][0]{$-6$}
\psfrag{D}[Br][Br][0.7][0]{$-4$}
\psfrag{E}[Br][Br][0.7][0]{$-2$}
\psfrag{F}[Br][Br][0.7][0]{$0$}
\psfrag{G}[Br][Br][0.7][0]{$2$}
\psfrag{H}[Br][Br][0.7][0]{$4$}
\psfrag{I}[Br][Br][0.7][0]{$6$}
\psfrag{J}[Br][Br][0.7][0]{$8$}
\psfrag{K}[Br][Br][0.7][0]{$10$}
\psfrag{Q}[Bl][Bl][0.7][0]{$\widetilde{L}(p)$}
\psfrag{R}[Bl][Bl][0.7][0]{$\widetilde{L}(p)+K$}
\psfrag{X}[Bl][Bl][0.8][0]{$p=1$}
\psfrag{Y}[Bc][Bc][0.8][0]{$p^\star=0.707$}
\includegraphics[width=5.4cm]{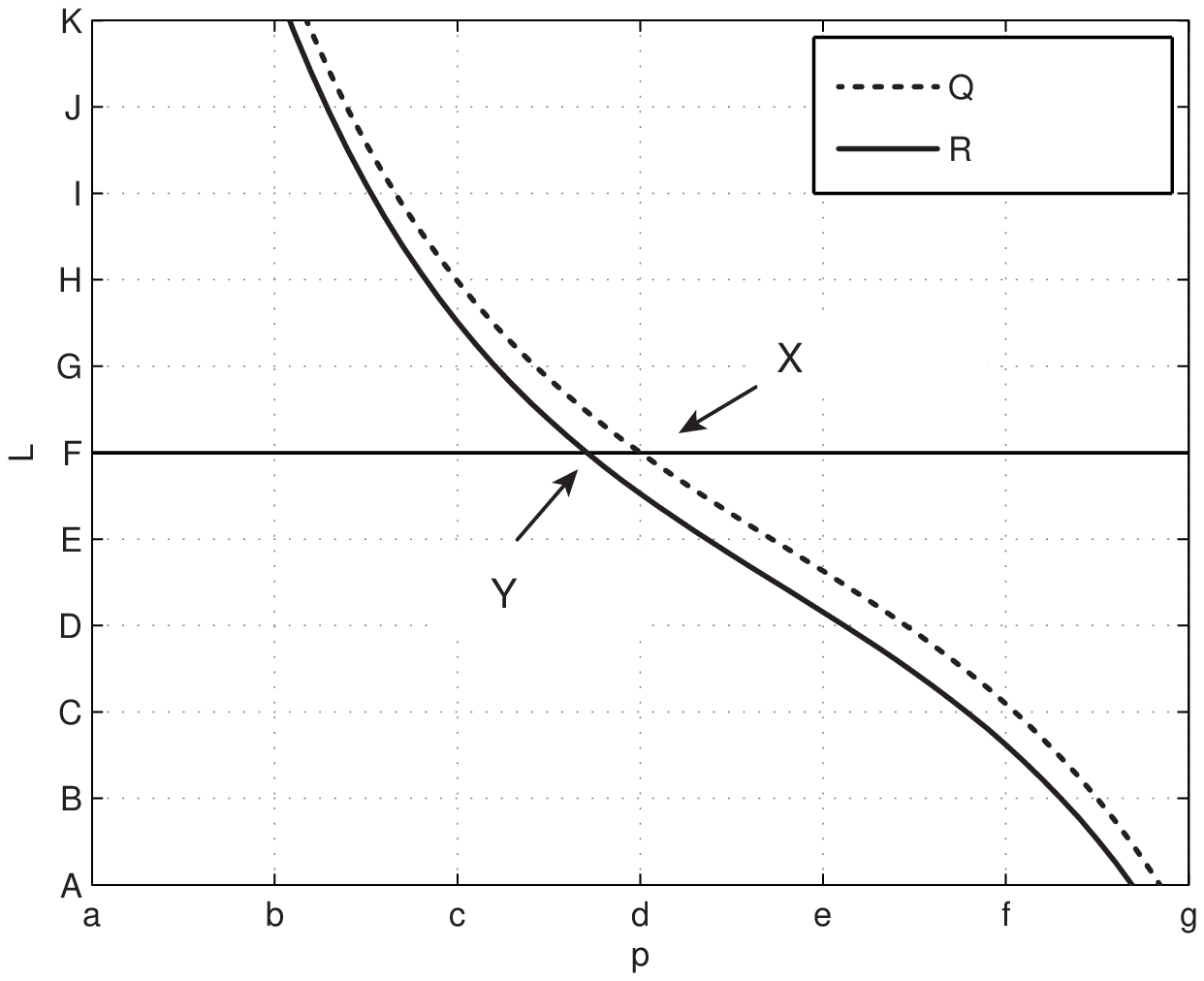}
}
\caption{Example for $n=3$: $x_1=1$, $x_2=2$, $x_3=3$. Inflection point $p^\star=0.707$.}\label{fig:fig2}
\end{figure}

We will now take use of Corollary~\ref{cor:ineq} for $n=3$. By expanding Eq.~\eqref{eq:ineq} we obtain 
\begin{multline*}\textstyle
\frac{1}{\left(x_1^p+x_2^p+x_3^p\right)\left(x_1^{p-1}+x_2^{p-1}+x_3^{p-1}\right)}\!\left(\!(x_1x_2)^{p-1}(x_1\!-\!x_2)\log\frac{x_1}{x_2}+(x_1x_3)^{p-1}(x_1\!-\!x_3)\log\frac{x_1}{x_3}+\right.\\
\textstyle\left.+(x_2x_3)^{p-1}(x_2\!-\!x_3)\log\frac{x_2}{x_3}\right)\geq 0.
\end{multline*}
After proper rearranging, we find that the following inequalities must hold.
\begin{subequations}
\begin{align}
\left(\frac{x_2}{x_3}\right)^{p-1}(x_1-x_2)\log\frac{x_2}{x_1}+\left(\frac{x_2}{x_1}\right)^{p-1}(x_2-x_3)\log\frac{x_3}{x_2}\leq (x_1-x_3)\log\frac{x_1}{x_3},\label{subeq:1}\\
\left(\frac{x_1}{x_3}\right)^{p-1}(x_1-x_2)\log\frac{x_2}{x_1}+\left(\frac{x_1}{x_2}\right)^{p-1}(x_1-x_3)\log\frac{x_3}{x_1}\leq(x_2-x_3)\log\frac{x_2}{x_3},\\
\left(\frac{x_3}{x_2}\right)^{p-1}(x_1-x_3)\log\frac{x_3}{x_1}+\left(\frac{x_3}{x_1}\right)^{p-1}(x_2-x_3)\log\frac{x_3}{x_2}\leq (x_1-x_2)\log\frac{x_1}{x_2}.\label{subeq:3}
\end{align}
\end{subequations}
By plugging Eqs.~\eqref{subeq:1}--\eqref{subeq:3} into Eq.~\eqref{eq:tildLder}, we find that
\[{\widetilde{L}}^\prime(p)\leq \log\frac{x_1}{x_3}\log\frac{x_2}{x_1}\log\frac{x_2}{x_3}\big((x_1-x_3)-(x_2-x_3)-(x_1-x_2)\big)=0,\]
thus $\widetilde{L}(p)$ is strictly decreasing for $x_1\ne x_2 \ne x_3$ and therefore also $L(p)$ has exactly one inflection point. 

Furthermore, from Eq.~\eqref{eq:n3} we see that for $p=1$, $L^{\prime\prime}(1)=\frac{1}{27}K$. Since the constant~$K$ (cf.~Eq.~\eqref{eq:K}) is always different from zero unless $x_1= x_2= x_3$, the inflection point~$p^\star$ must be different from $p=1$. Moreover, if $K<0$ then $p^\star<1$, and if $K>0$ then $p^\star>1$ (see the example on Fig.~\ref{fig:fig2} with $x_1=1$, $x_2=2$, and $x_3=3$).
\end{proof}

\vspace{0.3cm}

Unlike the previous case of $n=2$, the values of $x_i$ influence the location of the inflection point. We observe that the closer the values $x_i$ are to each other (a small variance), the further from $p=1$ the inflection point may be. Conversely, the bigger the variance, the closer to $p=1$ is the inflection point. Intuitively, by looking at $\widetilde{L}(p)$ one can see that the bigger the exponential terms are, e.g., $\left(\frac{x_2}{x_3}\right)^{p-1}\!\!-\left(\frac{x_1}{x_3}\right)^{p-1}$, the less ``freedom'' in choosing $p$ we have in order to $L^{\prime\prime}(p)=0$. Therefore, in this case, the inflection point must lie close to $p=1$.

\vspace{-0.25cm}

\subsection[Arbitrary n]{Arbitrary $n$}

As mentioned before the analysis for larger $n$ becomes more challenging. Already for $n=4$ the equivalent function to function $\widetilde{L}(p)$ may not be monotonous anymore, thus, there may be more inflection points.

From simulations we indeed found an example with three inflection points in case $n=4$ (see Fig.~\ref{fig:fig3}). Observe that the largest inflection point may be very far from $p=1$. We believe that this is the maximum possible number of inflection points for case $n=4$, however, a rigorous proof is missing at the moment. Nevertheless, we here provide a~simple combinatorial upper bound on the number of inflection points even in the general case.

\begin{figure}[t!]\centering
\subfloat[][Lehmer mean function $L(p)$.]{\centering
\psfrag{p}[Bc][Bc][0.8][0]{$p$}
\psfrag{L}[Bc][Bc][0.8][0]{\raisebox{0.6cm}{$L(p)$}}
\psfrag{a}[Bc][Bc][0.7][0]{$-500$}
\psfrag{b}[Bc][Bc][0.7][0]{$0$}
\psfrag{c}[Bc][Bc][0.7][0]{$500$}
\psfrag{d}[Bc][Bc][0.7][0]{$1000$}
\psfrag{e}[Bc][Bc][0.7][0]{$1500$}
\psfrag{f}[Bc][Bc][0.7][0]{$2000$}
\psfrag{g}[Bc][Bc][0.7][0]{$2500$}
\psfrag{A}[Br][Br][0.7][0]{$0.95$}
\psfrag{B}[Br][Br][0.7][0]{$0.96$}
\psfrag{C}[Br][Br][0.7][0]{$0.97$}
\psfrag{D}[Br][Br][0.7][0]{$0.98$}
\psfrag{E}[Br][Br][0.7][0]{$0.99$}
\psfrag{F}[Br][Br][0.7][0]{$1$}
\psfrag{G}[Br][Br][0.7][0]{$1.01$}
\psfrag{H}[Br][Br][0.7][0]{$1.02$}
\psfrag{I}[Br][Br][0.7][0]{$1.03$}
\psfrag{J}[Br][Br][0.7][0]{$1.04$}
\includegraphics[width=5.4cm]{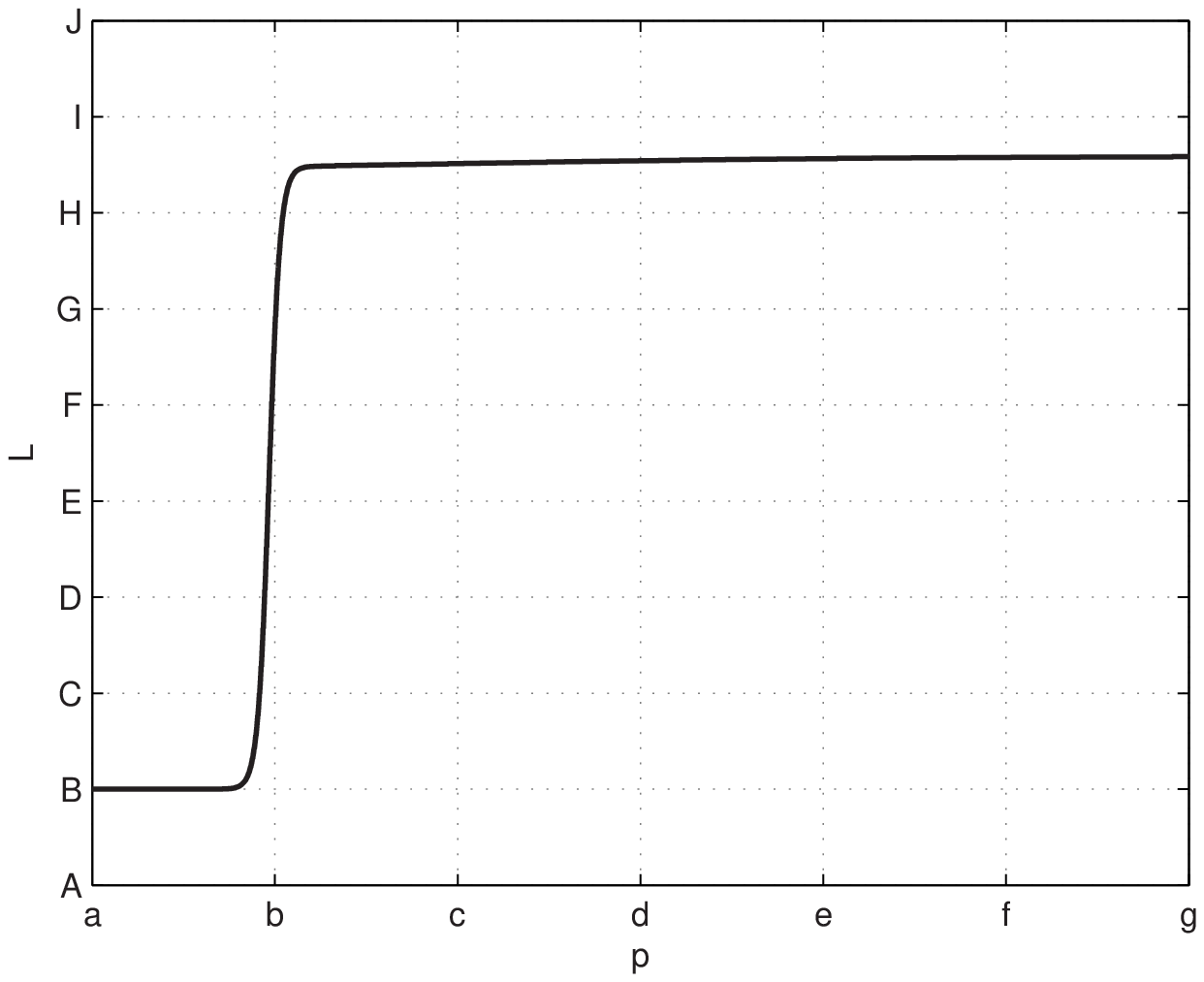}
}\hspace{0.7cm}
\subfloat[][Second derivative of function $L(p)$.]{\centering
\psfrag{p}[Bc][Bc][0.8][0]{$p$}
\psfrag{L}[Bc][Bc][0.8][0]{\raisebox{0.6cm}{$L^{\prime\prime}(p)$}}
\psfrag{a}[Bc][Bc][0.7][0]{$-500$}
\psfrag{b}[Bc][Bc][0.7][0]{$0$}
\psfrag{c}[Bc][Bc][0.7][0]{$500$}
\psfrag{d}[Bc][Bc][0.7][0]{$1000$}
\psfrag{e}[Bc][Bc][0.7][0]{$1500$}
\psfrag{f}[Bc][Bc][0.7][0]{$2000$}
\psfrag{g}[Bc][Bc][0.7][0]{$2500$}
\psfrag{A}[Br][Br][0.7][0]{$-1$}
\psfrag{B}[Br][Br][0.7][0]{$-0.8$}
\psfrag{C}[Br][Br][0.7][0]{$-0.6$}
\psfrag{D}[Br][Br][0.7][0]{$-0.4$}
\psfrag{E}[Br][Br][0.7][0]{$-0.2$}
\psfrag{F}[Br][Br][0.7][0]{$0$}
\psfrag{G}[Br][Br][0.7][0]{$0.2$}
\psfrag{H}[Br][Br][0.7][0]{$0.4$}
\psfrag{I}[Br][Br][0.7][0]{$0.6$}
\psfrag{J}[Br][Br][0.7][0]{$0.8$}
\psfrag{K}[Br][Br][0.7][0]{$1$}
\psfrag{M}[Bc][Bc][0.6][0]{$\times 10^{-9}$}
\psfrag{X}[Bl][Bl][0.7][0]{$p^\star_1=-15.8$}
\psfrag{Y}[Bl][Bl][0.7][0]{$p^\star_2=203.9$}
\psfrag{Z}[Bl][Bl][0.7][0]{$p^\star_3=401.4$}
\includegraphics[width=5.4cm]{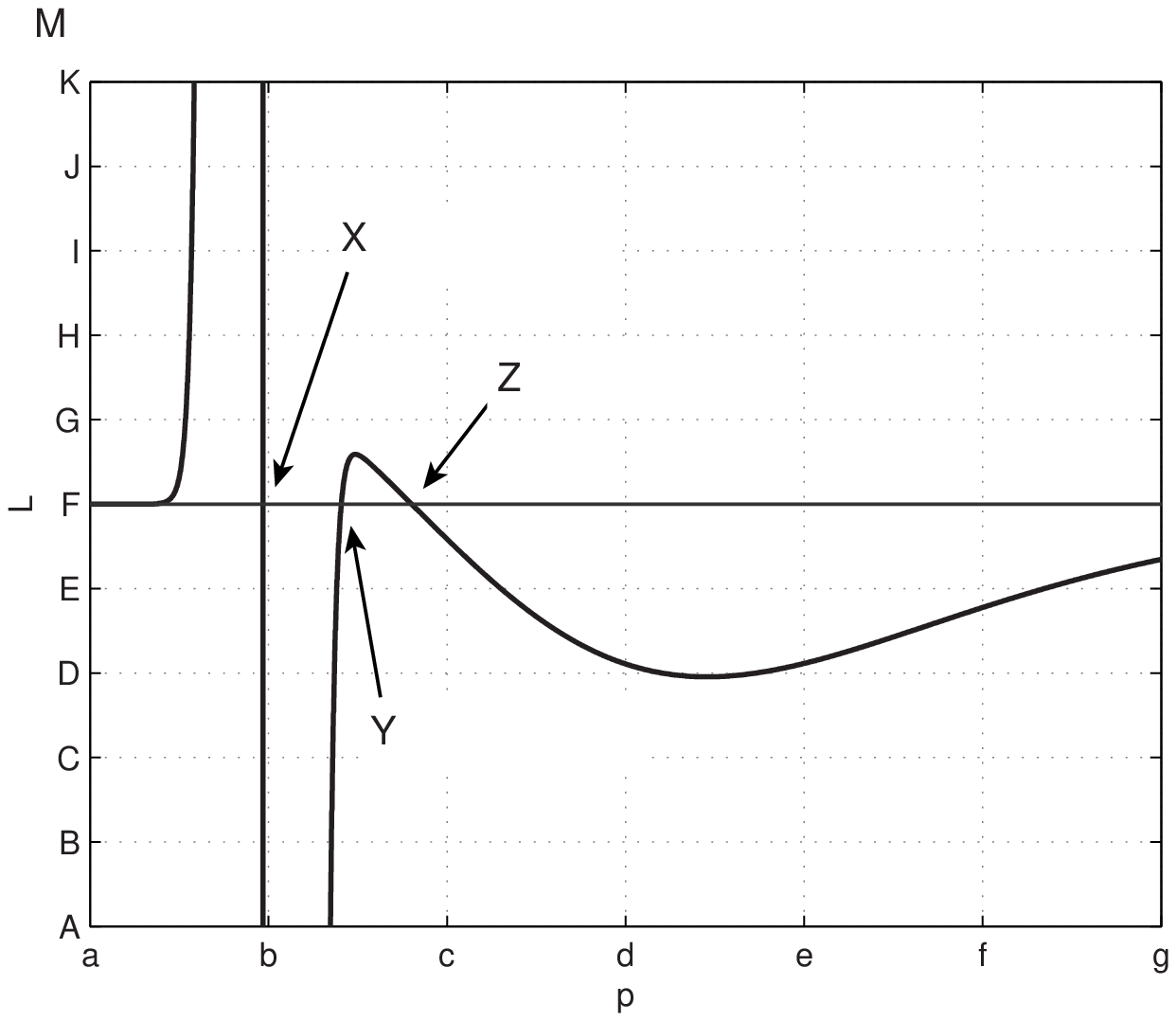}
}
\caption{Example of Lehmer mean function with three inflection points, $n=~4$. $x_i=~\{1.0259,1.0241,1.0244,0.96\}$, with inflection points $p^\star_1=-15.8075$, $p^\star_2=~203.9186$, $p^\star_3=401.3897$.}\label{fig:fig3}
\end{figure}

\begin{theorem}
The Lehmer mean function~\eqref{eq:lehmf} has at most $J=\frac{n(n+4)(n-1)}{6}-1$ inflection points, respectively, if $J$ is an even number, then it has at most $J-1$ inflection points.
\end{theorem}

\begin{proof}
Observe that the expression (cf. Eq.~\eqref{eq:secder}) 
\[\textstyle\left[\frac{\sum_{i=1}^n x_i^p(\log x_i)^2}{\sum_{i=1}^n x_i^p}-\frac{\sum_{i=1}^n x_i^{p-1}(\log x_i)^2}{\sum_{i=1}^n x_i^{p-1}}-2\frac{\sum_{i=1}^n x_i^{p-1}\log x_i}{\sum_{i=1}^n x_i^{p-1}}\left(\frac{\sum_{i=1}^n x_i^p\log x_i}{\sum_{i=1}^n x_i^p}-\frac{\sum_{i=1}^n x_i^{p-1}\log x_i}{\sum_{i=1}^n x_i^{p-1}}\right)\right],\] after putting the terms over a common denominator, is an exponential polynomial of order 3 with $n$ terms~\cite{Dreibelbis10}. After expanding the summations there will be exactly $\binom{n+3-1}{3}=\frac{(n+2)(n+1)n}{6}$ terms of the form $\left(x_{i_1}^{q_1}x_{i_2}^{q_2}x_{i_3}^{q_3}\right)^p$ such that $q_1+q_2+q_3=3$, where $\{i_1,i_2,i_3\}\subset\{1,2,\dots,n\}$. We further notice that all terms of the form $x_i^{3p}$ cancel out which leads to $N=\frac{(n+2)(n+1)n}{6}-n=\frac{n(n+4)(n-1)}{6}$ terms.

We know that an exponential polynomial, i.e., a polynomial of the form\break $P(x)=~\sum_{j=1}^Nc_ja_j^x$, can have at most $N-1$ zeros, in case $a_j>0$~\cite[Corollary~3.2]{Jameson06}.

Thus, $L^{\prime\prime}(p)$ has at most $J=\frac{n(n+4)(n-1)}{6}-1$ inflection points. Furthermore, from Lemma~\ref{lem:odd}, $J$ must be an odd number.
\end{proof}
\vspace{0.2cm}

This directly shows that if $n=2$, there can be at most one inflection point, as we showed in Sec.~\ref{subsec:n2}. For $n=3$, this bound gives us at most five inflection points, which is far from one inflection point as we proved in Sec.~\ref{subsec:n3}. For $n=4$, this bound gives us at most 15 inflection points. However, from the vast number of simulations, we believe that there can be at most three inflection points. We can thus see that this bound is very~loose.

\section{Conclusions}
We proved that the Lehmer mean function has exactly one inflection point in case $n=2$ and $n=3$. We further showed that in general the Lehmer mean function may have more inflections points and we provided a simple bound on the number of these inflection points. However, this bound is very loose and we believe that the number of inflection points must be less than or equal to $n$ for any $n$. 

We note that the convexity analysis of the Lehmer mean function is, in some sense, easier than the analysis of power mean functions, however, it is still a challenging task. New approaches and ideas are needed for a deeper analysis. For example, we have observed that the number of inflection points and their location is influenced by the variance of the values $x_i$. Typically, the closer to each other the values are, the less ``stable'' the function is, meaning that the inflection point may be far from one, and also the number of inflection points may grow. Conversely, if the values $x_i$ are far from each other, the function remains ``stable'', meaning that there is only one inflection point, which is close to one. To analyze this dependence remains, however, a challenging open issue.

\vspace{0.7cm}

\emph{Acknowledgements.} The author wishes to thank Prof. Markus Rupp for valuable comments and suggestions, especially for the idea for Lemma~\ref{lem:odd}. This work was supported by the Austrian Science Fund (FWF) under project grants S10611-N13 within the National Research Network SISE.

\bibliographystyle{siam}
\bibliography{references}
\clearpage
\end{document}